\documentclass[11pt,a4paper]{article}

\usepackage{inputenc}
\usepackage{amsmath}
\usepackage{bm}
\usepackage{bbold}
\usepackage{amsthm}
\usepackage{enumerate}

\usepackage{hyperref}

\title{Direct Solution to Constrained Tropical Optimization Problems with Application to\\ Project Scheduling}

\author{N. Krivulin\thanks{Faculty of Mathematics and Mechanics, Saint Petersburg State University, 28 Universitetsky Ave., Saint Petersburg, 198504, Russia, 
nkk@math.spbu.ru.}
}

\date{}

\newtheorem{theorem}{Theorem}
\newtheorem{lemma}[theorem]{Lemma}
\newtheorem{corollary}[theorem]{Corollary}

\begin{document}

\maketitle

\begin{abstract}
We examine a new optimization problem formulated in the tropical mathematics setting as a further extension of certain known problems. The problem is to minimize a nonlinear objective function, which is defined on vectors over an idempotent semifield by using multiplicative conjugate transposition, subject to inequality constraints. As compared to the known problems, the new one has a more general objective function and additional constraints. We provide a complete solution in an explicit form to the problem by using an approach that introduces an auxiliary variable to represent the values of the objective function, and then reduces the initial problem to a parametrized vector inequality. The minimum of the objective function is evaluated by applying the existence conditions for the solution of this inequality. A complete solution to the problem is given by solving the parametrized inequality, provided the parameter is set to the minimum value. As a consequence, we obtain solutions to new special cases of the general problem. To illustrate the application of the results, we solve a real-world problem drawn from time-constrained project scheduling, and offer a representative numerical example. 
\\

\\


\end{abstract}

\section{Introduction}

Tropical optimization problems constitute an important research and application domain of tropical mathematics. As an applied mathematical discipline that concentrates on the theory and methods of semirings with idempotent addition, tropical (idempotent) mathematics dates back to the works of \cite{Pandit1961Anew,Cuninghamegreen1962Describing,Hoffman1963Onabstract,Vorobjev1963Theextremal} and \cite{Romanovskii1964Asymptotic}, at least two of which \cite{Cuninghamegreen1962Describing,Hoffman1963Onabstract} have been motivated and illustrated by optimization problems.

Many subsequent publications that contributed to the development of tropical mathematics, including the monographs by \cite{Cuninghamegreen1979Minimax,Zimmermann1981Linear,Kolokoltsov1997Idempotent,Gondran2008Graphs,Butkovic2010Maxlinear,Maclagan2015Introduction}, and a number of contributed papers, were concerned with optimization problems, most of which have been drawn from real-world applications in operations research and management science.

Multidimensional tropical optimization problems are generally formulated in the tropical mathematics setting to minimize or maximize linear and nonlinear functions defined on vectors over idempotent semifields (semirings with multiplicative inverses). The problems may include constraints given by linear and nonlinear equalities and inequalities. Many of the problems that come from real-world applications and, at the same time, admit solutions in the framework of tropical mathematics have nonlinear objective functions defined through multiplicative conjugate transposition of vectors (see, e.g., an overview in \cite{Krivulin2015Amultidimensional}).

There are problems with objective functions that involve the tropical algebraic product $\bm{x}^{-}\bm{A}\bm{x}$, where $\bm{A}$ is a given square matrix, $\bm{x}$ is the unknown vector, and $\bm{x}^{-}$ is the multiplicative conjugate transpose of $\bm{x}$. These functions appear in various applications in operations research and management science, including problems in project (machine) scheduling \cite{Cuninghamegreen1962Describing,Cuninghamegreen1979Minimax,Superville1978Various,Krivulin2015Extremal,Krivulin2015Tropicaloptimization,Krivulin2015Tropical}, location analysis \cite{Zimmermann1992Optimization,Hudec1993Aservice,Krivulin2011Anextremal}, and decision making \cite{Elsner2004Maxalgebra,Elsner2010Maxalgebra,Gursoy2013Theanalytic}, to name only a few.

The problem of minimizing the product in question was examined in early works \cite{Cuninghamegreen1962Describing,Engel1975Diagonal,Superville1978Various} by using conventional mathematical techniques. It was shown that the minimum in the problem is equal to the tropical spectral radius of the matrix $\bm{A}$, and attained at the corresponding tropical eigenvectors of this matrix. Later, the problem was formulated in the framework of tropical mathematics in \cite{Cuninghamegreen1979Minimax}, where a complete solution was proposed by reducing to a linear programming problem. Solutions based on tropical mathematics were derived in \cite{Elsner2004Maxalgebra,Elsner2010Maxalgebra}. The results of \cite{Elsner2010Maxalgebra} included an implicit description of a complete solution in the form of a vector inequality, and provided a computational procedure to solve the inequality. Finally, complete solutions in terms of tropical mathematics to both the problem and its generalizations, which have objective functions of an extended form as well as additional constraints, were given in \cite{Krivulin2014Aconstrained,Krivulin2015Extremal,Krivulin2015Amultidimensional}.

In this paper, we consider a new rather general optimization problem, which includes known problems as special cases. We provide a complete solution in an explicit form on the basis of the approach developed in \cite{Krivulin2014Aconstrained,Krivulin2015Extremal,Krivulin2015Amultidimensional}, which introduces an additional variable to represent the values of the objective function, and then reduces the initial problem to a parametrized vector inequality. The minimum of the objective function is evaluated by using the solution existence conditions for the inequality. A complete solution to the problem is given by the solutions of the parametrized inequality, provided the parameter is set to the minimum value. We discuss the computational complexity of the result to show that the solution can be obtained in polynomial time. As a consequence, we propose solutions to new special cases of the general problem.

We apply the results obtained to derive a new complete solution of a real-world problem that was drawn from project scheduling (see, e.g., \cite{Demeulemeester2002Project,Neumann2003Project,Tkindt2006Multicriteria} for further details on optimal scheduling), and also served to motivate the study. We consider a project that consists of activities operating in parallel under temporal constraints in various forms, including release dates and time windows. For each activity, the flow-time is defined to be the time interval between its initiation and completion. The objective is to find an optimal schedule that minimizes the maximum flow time over all activities. This problem is an extended version of that in \cite{Krivulin2015Extremal}, where a less complicated system of temporal constraints is considered. To illustrate the solution obtained for the problem, and the computational technique implemented by the solution, we present a representative numerical example.

Note that the problem under examination can be formulated as a linear program, and then solved by one of the known solution techniques of linear programming. However, these techniques usually take the form of iterative algorithms, and do not generally guarantee an explicit closed-form solution. Unlike the algorithmic approaches, the proposed solution provides direct results in a compact vector form suitable for further analysis and practical use. Considering, in addition, that the new solution can be calculated in polynomial time, it can certainly serve as a helpful complement and supplement to existing solutions. 

The paper is organized as follows. Section~\ref{S-BDNO} includes definitions and notation to be used in the subsequent sections. In Section~\ref{S-PR}, we present some preliminary results, including a binomial identity for matrices and the solution to linear inequalities. The main result is provided in Section~\ref{S-SOP}, where we first offer examples of known optimization problems, then formulate and solve a new general problem, discuss the computational complexity of the solution, and finally, give solutions to new special cases of the general problem. Section~\ref{S-APS} contains an application of the results in project scheduling, and concludes with a numerical example.

\section{Basic definitions, notation and observations}
\label{S-BDNO}

We start with a short introduction in the context of tropical (idempotent) algebra to offer a unified and self-contained framework for the formulation and solution of tropical optimization problems in the rest of the paper. Below, we follow the notation and results in \cite{Krivulin2014Aconstrained,Krivulin2015Extremal,Krivulin2015Amultidimensional}, which form a useful basis for the analysis and solution of the problems under study in a systematic manner and in a compact closed form. Further details on tropical mathematics at both introductory and advanced levels can be found in \cite{Baccelli1993Synchronization,Kolokoltsov1997Idempotent,Golan2003Semirings,Heidergott2006Maxplus,Gondran2008Graphs,Butkovic2010Maxlinear,Maclagan2015Introduction}.

\subsection{Idempotent semifield}

An idempotent semifield is an algebraic structure $(\mathbb{X},\oplus,\otimes,\mathbb{0},\mathbb{1})$, where $\mathbb{X}$ is a nonempty set, $\oplus$ and $\otimes$ are binary operations, called addition and multiplication, $\mathbb{0}$ and $\mathbb{1}$ are distinct elements in $\mathbb{X}$, called zero and one, such that $(\mathbb{X},\oplus,\mathbb{0})$ is an idempotent commutative monoid, $(\mathbb{X}\setminus\{\mathbb{0}\},\otimes,\mathbb{1})$ is an Abelian group, and multiplication distributes over addition. 

The semifield has idempotent addition, which implies that $x\oplus x=x$ for each $x\in\mathbb{X}$, and invertible multiplication, which allows each nonzero $x$ to have its multiplicative inverse $x^{-1}$ such that $x\otimes x^{-1}=\mathbb{1}$. 

Idempotent addition induces a partial order on $\mathbb{X}$ such that $x\leq y$ if and only if $x\oplus y=y$. It follows from this definition that $x\leq x\oplus y$ and $y\leq x\oplus y$. Furthermore, both operations $\oplus$ and $\otimes$ are monotone, which implies that the inequality $x\leq y$ yields $x\oplus z\leq y\oplus z$ and $x\otimes z\leq y\otimes z$ for all $z$. The inversion is antitone, which means that the inequality $x\leq y$ results in $x^{-1}\geq y^{-1}$ for nonzero $x$ and $y$. Finally, the inequality $x\oplus y\leq z$ is equivalent to the two inequalities $x\leq z$ and $y\leq z$.

It is assumed that the partial order can be extended to a linear one to take the semifield as linearly ordered. The relation symbols and the optimization objectives are considered below in terms of this order.

Integer powers are routinely used as shorthand for iterated multiplication such that $x^{0}=\mathbb{1}$ and $x^{m}=x\otimes x^{m-1}$ for all $x\in\mathbb{X}$ and integer $m\geq1$. Moreover, it is assumed that the equation $x^{m}=a$ has a solution for any $a\in\mathbb{X}$ and positive integer $m$, which extends the power notation to rational exponents, and thus makes the semifield algebraically complete (radicable). In the expressions that follow, the multiplication sign $\otimes$ is omitted for brevity. 

Examples of the semifield include $\mathbb{R}_{\max,+}=(\mathbb{R}\cup\{-\infty\},\max,+,-\infty,0)$ and $\mathbb{R}_{\min,\times}=(\mathbb{R}_{+}\cup\{+\infty\},\min,\times,+\infty,1)$, where $\mathbb{R}$ is the set of real numbers and $\mathbb{R}_{+}=\{x>0|x\in\mathbb{R}\}$, to list only a few.

The semifield $\mathbb{R}_{\max,+}$ is equipped with addition and multiplication defined, respectively, as $\max$ and $+$. Furthermore, the number $-\infty$ is taken as zero, and $0$ is as one. Each $x\in\mathbb{R}$ has the inverse $x^{-1}$, which corresponds to the opposite number $-x$ in the usual notation. The power $x^{y}$ exists for any $x,y\in\mathbb{R}$ and coincides with the ordinary arithmetic product $xy$. The order defined by idempotent addition is consistent with the conventional total order on $\mathbb{R}$. 

In $\mathbb{R}_{\min,\times}$, we have $\oplus=\min$, $\otimes=\times$, $\mathbb{0}=+\infty$ and $\mathbb{1}=1$. The inversion and exponentiation notations have the usual meaning. The relation $\leq$ defines an order that is opposite to the standard linear order on $\mathbb{R}$.

\subsection{Matrices and vectors}

The set of matrices of $m$ rows and $n$ columns over $\mathbb{X}$ is denoted $\mathbb{X}^{m\times n}$. A matrix with all entries equal to $\mathbb{0}$ is the zero matrix denoted by $\bm{0}$. A matrix is row- (column-) regular, if it has no zero rows (columns).

Matrix addition and multiplication, and scalar multiplication follow the usual rules with the scalar operations $\oplus$ and $\otimes$ in place of the ordinary addition and multiplication. The above inequalities, which represent properties of the scalar operations, are extended entry-wise to matrix inequalities.

For any matrix $\bm{A}\in\mathbb{X}^{m\times n}$, its transpose is the matrix $\bm{A}^{T}\in\mathbb{X}^{n\times m}$.

The square matrices of order $n$ form the set denoted by $\mathbb{X}^{n\times n}$. A square matrix having $\mathbb{1}$ along the diagonal and $\mathbb{0}$ elsewhere is the identity matrix denoted by $\bm{I}$. For any square matrix $\bm{A}$, the nonnegative integer power is defined as $\bm{A}^{0}=\bm{I}$ and $\bm{A}^{m}=\bm{A}\bm{A}^{m-1}$ for all integers $m\geq1$.

The trace of a matrix $\bm{A}=(a_{ij})\in\mathbb{X}^{n\times n}$ is given by
$$
\mathop\mathrm{tr}\bm{A}
=
a_{11}\oplus\cdots\oplus a_{nn}
=
\bigoplus_{i=1}^{n}a_{ii}.
$$

The trace possesses the usual properties given by the equalities
$$
\mathop\mathrm{tr}(\bm{A}\oplus\bm{B})
=
\mathop\mathrm{tr}\bm{A}
\oplus
\mathop\mathrm{tr}\bm{B},
\qquad
\mathop\mathrm{tr}(\bm{A}\bm{B})
=
\mathop\mathrm{tr}(\bm{B}\bm{A}),
\qquad
\mathop\mathrm{tr}(x\bm{A})
=
x\mathop\mathrm{tr}\bm{A},
$$
which are valid for any matrices $\bm{A},\bm{B}\in\mathbb{X}^{n\times n}$ and scalar $x\in\mathbb{X}$. 

A matrix with only one column (row) is a column (row) vector. In what follows, all vectors are column vectors unless otherwise indicated. The set of column vectors of order $n$ is denoted $\mathbb{X}^{n}$.

A vector is regular if it has only nonzero elements. Let $\bm{x}\in\mathbb{X}^{n}$ be a regular vector and $\bm{A}\in\mathbb{X}^{n\times n}$ be a row-regular matrix. Then, the result of the multiplication $\bm{A}\bm{x}$ is a regular vector. If the matrix $\bm{A}$ is column-regular, then the row vector $\bm{x}^{T}\bm{A}$ is regular as well.

For any nonzero vector $\bm{x}\in\mathbb{X}^{n}$, its multiplicative conjugate transpose is the row vector $\bm{x}^{-}=(x_{i}^{-})$, where $x_{i}^{-}=x_{i}^{-1}$ if $x_{i}\ne\mathbb{0}$, and $x_{i}^{-}=\mathbb{0}$ otherwise.

The conjugate transposition exhibits some significant properties to be used later. Specifically, if $\bm{x}$ and $\bm{y}$ are regular vectors of the same order, then the inequality $\bm{x}\leq\bm{y}$ implies $\bm{x}^{-}\geq\bm{y}^{-}$ and vice versa. Furthermore, for any nonzero vector $\bm{x}$, the equality $\bm{x}^{-}\bm{x}=\mathbb{1}$ holds. Finally, if the vector $\bm{x}$ is regular, then the matrix inequality $\bm{x}\bm{x}^{-}\geq\bm{I}$ is also valid. 

A scalar $\lambda\in\mathbb{X}$ is an eigenvalue of a matrix $\bm{A}\in\mathbb{X}^{n\times n}$, if there exists a nonzero vector $\bm{x}\in\mathbb{X}^{n}$ such that $\bm{A}\bm{x}=\lambda\bm{x}$. The maximum eigenvalue is referred to as the spectral radius of $\bm{A}$, and given by \cite{Cuninghamegreen1962Describing,Vorobjev1963Theextremal,Romanovskii1964Asymptotic}
$$
\lambda
=
\mathop\mathrm{tr}\nolimits\bm{A}\oplus\cdots\oplus\mathop\mathrm{tr}\nolimits^{1/n}(\bm{A}^{n})
=
\bigoplus_{m=1}^{n}\mathop\mathrm{tr}\nolimits^{1/m}(\bm{A}^{m}).
$$

\section{Preliminary results}
\label{S-PR}

We now offer some auxiliary results to be used in the subsequent analysis of optimization problems. We start with binomial identities for square matrices, and then describe solutions to linear vector inequalities.

The inequalities are examined using somewhat different techniques and notation by many authors, including \cite{Vorobjev1963Theextremal,Cuninghamegreen1979Minimax,Zimmermann1981Linear,Baccelli1993Synchronization,Gondran2008Graphs}. Below, we offer solutions given in a compact vector form that provides a natural basis for solving the optimization problems in a straightforward and concise manner.  

\subsection{Binomial identities}

Let $\bm{A}$ and $\bm{B}$ be square matrices of the same order, and $m$ be a positive integer. Then, the following binomial identity clearly holds:
$$
(\bm{A}\oplus\bm{B})^{m}
=
\bigoplus_{k=1}^{m}\mathop{\bigoplus\hspace{1.3em}}_{i_{0}+i_{1}+\cdots+i_{k}=m-k}\bm{B}^{i_{0}}(\bm{A}\bm{B}^{i_{1}}\bm{A}\bm{B}^{i_{2}}\cdots\bm{A}\bm{B}^{i_{k}})
\oplus
\bm{B}^{m}.
$$

As an extension of this identity, we derive the following results. First, after summation over all $m$ and rearrangement of the output to collect terms of like number of cofactors $\bm{A}$, we obtain the matrix equality
\begin{equation}
\bigoplus_{k=1}^{m}(\bm{A}\oplus\bm{B})^{k}
=
\bigoplus_{k=1}^{m}\mathop{\bigoplus\hspace{0.2em}}_{0\leq i_{0}+i_{1}+\cdots+i_{k}\leq m-k}\bm{B}^{i_{0}}(\bm{A}\bm{B}^{i_{1}}\cdots\bm{A}\bm{B}^{i_{k}})
\oplus
\bigoplus_{k=1}^{m}\bm{B}^{k}.
\label{E-ABk}
\end{equation}

Furthermore, by applying the trace and by using its properties, we rewrite \eqref{E-ABk} in the form of the scalar equality
\begin{equation}
\bigoplus_{k=1}^{m}\mathop\mathrm{tr}(\bm{A}\oplus\bm{B})^{k}
=
\bigoplus_{k=1}^{m}\mathop{\bigoplus\hspace{1.5em}}_{0\leq i_{1}+\cdots+i_{k}\leq m-k}\mathop\mathrm{tr}(\bm{A}\bm{B}^{i_{1}}\cdots\bm{A}\bm{B}^{i_{k}})
\oplus
\bigoplus_{k=1}^{m}\mathop\mathrm{tr}\bm{B}^{k}.
\label{E-trABk}
\end{equation}

Both identities \eqref{E-ABk} and \eqref{E-trABk} are used below to expand matrix expressions in evaluating the minimum of the objective function.

\subsection{Linear inequalities}

Suppose that, given a matrix $\bm{A}\in\mathbb{X}^{m\times n}$ and a regular vector $\bm{d}\in\mathbb{X}^{m}$, the problem is to find all vectors $\bm{x}\in\mathbb{X}^{n}$ that satisfy the inequality
\begin{equation}
\bm{A}\bm{x}
\leq
\bm{d}.
\label{I-Axd}
\end{equation}

A complete direct solution to the problem under fairly general conditions can be found in the following form (see, e.g., \cite{Krivulin2015Extremal}).
\begin{lemma}
\label{L-IxdA}
For any column-regular matrix $\bm{A}$ and regular vector $\bm{d}$, all solutions to \eqref{I-Axd} are given by
\begin{equation*}
\bm{x}
\leq
(\bm{d}^{-}\bm{A})^{-}.
%\label{I-xdA}
\end{equation*}
\end{lemma}

Furthermore, we consider the problem: given a matrix $\bm{A}\in\mathbb{X}^{n\times n}$ and a vector $\bm{b}\in\mathbb{X}^{n}$, find all regular vectors $\bm{x}\in\mathbb{X}^{n}$ that solve the inequality
\begin{equation}
\bm{A}\bm{x}\oplus\bm{b}
\leq
\bm{x}.
\label{I-Axbx}
\end{equation}

To describe a solution to inequality \eqref{I-Axbx} in a compact form, we introduce a function that maps each matrix $\bm{A}\in\mathbb{X}^{n\times n}$ onto the scalar
$$
\mathop\mathrm{Tr}(\bm{A})
=
\mathop\mathrm{tr}\bm{A}\oplus\cdots\oplus\mathop\mathrm{tr}\bm{A}^{n},
$$
and use the asterate operator (the Kleene star), which takes $\bm{A}$ to the matrix
$$
\bm{A}^{\ast}
=
\bm{I}\oplus\bm{A}\oplus\cdots\oplus\bm{A}^{n-1}.
$$

Presented below is a complete solution proposed in \cite{Krivulin2015Amultidimensional}.
\begin{theorem}
\label{T-Axbx}
For any matrix $\bm{A}$ and vector $\bm{b}$, the following statements hold:
\begin{enumerate}
\item If $\mathop\mathrm{Tr}(\bm{A})\leq\mathbb{1}$, then all regular solutions to inequality \eqref{I-Axbx} are given by $\bm{x}=\bm{A}^{\ast}\bm{u}$, where $\bm{u}$ is a regular vector such that $\bm{u}\geq\bm{b}$.
\item If $\mathop\mathrm{Tr}(\bm{A})>\mathbb{1}$, then there is no regular solution.
\end{enumerate}
\end{theorem}

To conclude this section, we present a solution to a system that combines inequality \eqref{I-Axbx} with an upper bound on the vector $\bm{x}$ in the form
\begin{equation}
\begin{aligned}
\bm{A}\bm{x}
\oplus
\bm{b}
&\leq
\bm{x},
\\
\bm{x}
&\leq
\bm{d}.
\end{aligned}
\label{I-Axbx-xd}
\end{equation}

By application of both Lemma~\ref{L-IxdA} and Theorem~\ref{T-Axbx}, we arrive at the next solution, which is also a direct consequence of the result obtained in \cite{Krivulin2014Aconstrained} for a slightly more general system.
\begin{lemma}
\label{L-Axbx-xd}
For any matrix $\bm{A}$, vector $\bm{b}$ and regular vector $\bm{d}$, we denote $\Delta=\mathop\mathrm{Tr}(\bm{A})\oplus\bm{d}^{-}\bm{A}^{\ast}\bm{b}$. Then, the following statements hold:
\begin{enumerate}
\item If $\Delta\leq\mathbb{1}$, then all regular solutions to system \eqref{I-Axbx-xd} are given by $\bm{x}=\bm{A}^{\ast}\bm{u}$, where $\bm{u}$ is a regular vector such that $\bm{b}\leq\bm{u}\leq(\bm{d}^{-}\bm{A}^{\ast})^{-}$.
\item If $\Delta>\mathbb{1}$, then there is no regular solution.
\end{enumerate}
\end{lemma}

\section{Solution to optimization problems}
\label{S-SOP}

In this section, we consider optimization problems involving the function $\bm{x}^{-}\bm{A}\bm{x}$, where $\bm{A}$ is a given matrix, and $\bm{x}$ is the unknown vector. The unconstrained minimization of this function is examined by different methods in various application contexts \cite{Cuninghamegreen1962Describing,Engel1975Diagonal,Superville1978Various,Cuninghamegreen1979Minimax,Elsner2004Maxalgebra,Elsner2010Maxalgebra}. Complete solutions to some constrained problems are proposed in \cite{Krivulin2014Aconstrained,Krivulin2015Amultidimensional,Krivulin2015Extremal} 

We present examples of both unconstrained and constrained problems, and then formulate and solve a new general constrained optimization problem. As a consequence, we offer solutions for some new special cases of the general problem.

The results are given in the context of an arbitrary idempotent semifield in a common form, which can be readily interpreted in terms of particular semifields. Specifically, for the semifield $\mathbb{R}_{\max,+}$, we replace $\oplus$ by $\max$ and $\otimes$ by $+$, and use the relation symbol $\leq$ in the usual sense. In the framework of $\mathbb{R}_{\min,\times}$, we put $\oplus=\min$ and $\otimes=\times$, and understand the symbol $\leq$ to indicate the order, which is opposite to the standard linear order on $\mathbb{R}$.

\subsection{Examples of optimization problems}

We start with an unconstrained problem that has the objective function written in a basic form. Given a matrix $\bm{A}\in\mathbb{X}^{n\times n}$, consider the problem to find regular vectors $\bm{x}\in\mathbb{X}^{n}$ that
\begin{equation}
\begin{aligned}
&
\text{minimize}
&&
\bm{x}^{-}\bm{A}\bm{x},
\end{aligned}
\label{P-minxAx}
\end{equation}

A solution to the problem can be provided by several ways (see, e.g., \cite{Krivulin2014Aconstrained,Krivulin2015Amultidimensional,Krivulin2015Extremal}), and takes the following form.
\begin{lemma}\label{L-minxAx}
Let $\bm{A}$ be a matrix with spectral radius $\lambda>\mathbb{0}$. Then, the minimum value in problem \eqref{P-minxAx} is equal to $\lambda$, and all regular solutions are given by
$$
\bm{x}
=
(\lambda^{-1}\bm{A})^{\ast}\bm{u},
\qquad
\bm{u}\in\mathbb{X}^{n}.
$$
\end{lemma}

Some extensions of problem \eqref{P-minxAx} were examined in \cite{Krivulin2014Aconstrained,Krivulin2015Amultidimensional,Krivulin2015Extremal}, where more general forms of the objective function are considered and/or further inequality constraints are added. Specifically, a problem with an extended function is solved in \cite{Krivulin2015Extremal}. Given a matrix $\bm{A}\in\mathbb{X}^{n\times n}$, vectors $\bm{p},\bm{q}\in\mathbb{X}^{n}$, and a scalar $r\in\mathbb{X}$, the problem is to obtain regular $\bm{x}\in\mathbb{X}^{n}$ that
\begin{equation}
\begin{aligned}
&
\text{minimize}
&&
\bm{x}^{-}\bm{A}\bm{x}\oplus\bm{x}^{-}\bm{p}\oplus\bm{q}^{-}\bm{x}\oplus r.
\end{aligned}
\label{P-xAxxpqxr}
\end{equation}

A complete direct solution to the problem is as follows.
\begin{theorem}\label{T-xAxxpqxr}
Let $\bm{A}$ be a matrix with spectral radius $\lambda>\mathbb{0}$, and $\bm{q}$ be a regular vector. Then, the minimum value in problem \eqref{P-xAxxpqxr} is equal to
\begin{equation*}
\mu
=
\lambda
\oplus
\bigoplus_{m=0}^{n-1}
(\bm{q}^{-}\bm{A}^{m}\bm{p})^{1/(m+2)}
\oplus
r,
%\label{E-mulambdaqApr}
\end{equation*}
and all regular solutions are given by
\begin{equation*}
\bm{x}
=
(\mu^{-1}\bm{A})^{\ast}\bm{u},
\qquad
\mu^{-1}\bm{p}
\leq
\bm{u}
\leq
\mu(\bm{q}^{-}(\mu^{-1}\bm{A})^{\ast})^{-}.
\end{equation*}
\end{theorem}

Suppose now that, given matrices $\bm{A},\bm{B}\in\mathbb{X}^{n\times n}$, and a vector $\bm{g}\in\mathbb{X}^{n}$, we need to find regular solutions $\bm{x}\in\mathbb{X}^{n}$ to the problem
\begin{equation}
\begin{aligned}
&
\text{minimize}
&&
\bm{x}^{-}\bm{A}\bm{x},
\\
&
\text{subject to}
&&
\bm{B}\bm{x}\oplus\bm{g}
\leq
\bm{x}.
\end{aligned}
\label{P-xAxBxgx}
\end{equation}

The next complete solution to the problem is provided in \cite{Krivulin2015Amultidimensional}.

\begin{theorem}\label{T-xAxBxgx}
Let $\bm{A}$ be a matrix with spectral radius $\lambda>\mathbb{0}$, and $\bm{B}$ a matrix with $\mathop\mathrm{Tr}(\bm{B})\leq\mathbb{1}$. Then, the minimum value in problem \eqref{P-xAxBxgx} is equal to
\begin{equation*}
\mu
=
\lambda
\oplus
\bigoplus_{k=1}^{n-1}\mathop{\bigoplus\hspace{1.2em}}_{1\leq i_{1}+\cdots+i_{k}\leq n-k}\mathop\mathrm{tr}\nolimits^{1/k}(\bm{A}\bm{B}^{i_{1}}\cdots\bm{A}\bm{B}^{i_{k}}),
\end{equation*}
and all regular solutions are given by
\begin{equation*}
\bm{x}
=
(\mu^{-1}\bm{A}\oplus\bm{B})^{\ast}\bm{u},
\qquad
\bm{u}
\geq
\bm{g}.
%\label{E-xthetaABu-ug}
\end{equation*}
\end{theorem}

Below, we offer a solution to a new problem that combine the objective function in \eqref{P-xAxxpqxr} with the extended set of constraints in \eqref{I-Axbx-xd}.

\subsection{New constrained optimization problem}
\label{S-NCOP}

This section includes a complete solution to a constrained problem, which presents an extended version of the problems considered above. We follow the approach developed in \cite{Krivulin2014Aconstrained,Krivulin2015Amultidimensional,Krivulin2015Extremal} to introduce an additional variable, which represents the minimum value of the objective function, and then to reduce the problem to an inequality, where the new variable plays the role of a parameter. 

Suppose that, given matrices $\bm{A},\bm{B}\in\mathbb{X}^{n\times n}$, vectors $\bm{p},\bm{q},\bm{g},\bm{h}\in\mathbb{X}^{n}$, and a scalar $r\in\mathbb{X}$, the problem is to find regular vectors $\bm{x}\in\mathbb{X}^{n}$ that
\begin{equation}
\begin{aligned}
&
\text{minimize}
&&
\bm{x}^{-}\bm{A}\bm{x}\oplus\bm{x}^{-}\bm{p}\oplus\bm{q}^{-}\bm{x}\oplus r,
\\
&
\text{subject to}
&&
\bm{B}\bm{x}\oplus\bm{g}
\leq
\bm{x},
\\
&&&
\bm{x}
\leq
\bm{h}.
\end{aligned}
\label{P-xAxxpqxr-Bxgx-xh}
\end{equation}

We start with some general remarks and useful notation. It immediately follows from Lemma~\ref{L-Axbx-xd} that the inequality constraints in \eqref{P-xAxxpqxr-Bxgx-xh} have regular solutions if and only if the condition $\mathop\mathrm{Tr}(\bm{B})\oplus\bm{h}^{-}\bm{B}^{\ast}\bm{g}\leq\mathbb{1}$ holds, which is itself equivalent to the two conditions $\mathop\mathrm{Tr}(\bm{B})\leq\mathbb{1}$ and $\bm{h}^{-}\bm{B}^{\ast}\bm{g}\leq\mathbb{1}$.

Clearly, the constraints can be rearranged to provide another representation of the problem in the form
\begin{equation}
\begin{aligned}
&
\text{minimize}
&&
\bm{x}^{-}\bm{A}\bm{x}\oplus\bm{x}^{-}\bm{p}\oplus\bm{q}^{-}\bm{x}\oplus r,
\\
&
\text{subject to}
&&
\bm{B}\bm{x}
\leq
\bm{x},
\\
&&&
\bm{g}
\leq
\bm{x}
\leq
\bm{h}.
\end{aligned}
\label{P-xAxxpqxr-Bxx-gxh}
\end{equation}

To describe the solution in a compact form, we introduce an auxiliary notation for large matrix sums. First, we define the matrices $\bm{S}_{0}=\bm{I}$ and
\begin{equation}
\bm{S}_{k}
=
\mathop{\bigoplus\hspace{1.2em}}_{0\leq i_{1}+\cdots+i_{k}\leq n-k}\bm{A}\bm{B}^{i_{1}}\cdots\bm{A}\bm{B}^{i_{k}},
\qquad
k=1,\ldots,n;
\label{E-Sk}
\end{equation}
and note that they satisfy the inequality $\bm{S}_{k}\geq\bm{A}^{k}$.

For a different type of sums, we introduce the notation $\bm{T}_{0}=\bm{B}^{\ast}$ and
\begin{equation}
\bm{T}_{k}
=
\mathop{\bigoplus\hspace{1.0em}}_{0\leq i_{0}+i_{1}+\cdots+i_{k}\leq n-k-1}
\bm{B}^{i_{0}}(\bm{A}\bm{B}^{i_{1}}\cdots\bm{A}\bm{B}^{i_{k}}),
\qquad
k=1,\ldots,n-1.
\label{E-Tk}
\end{equation}

It is easy to see that the matrices are related by the equality $\bm{S}_{k+1}=\bm{A}\bm{T}_{k}$, which is valid for all $k=0,1,\ldots,n-1$. Finally, note that, under the condition $\bm{B}=\bm{0}$, the matrices reduce to $\bm{S}_{k}=\bm{A}^{k}$ and $\bm{T}_{k}=\bm{A}^{k}$.

We are now in a position to offer a complete solution to problem \eqref{P-xAxxpqxr-Bxgx-xh}.
\begin{theorem}
\label{T-xAxxpqxr-Bxgx-xh}
Let $\bm{A}$ be a matrix with spectral radius $\lambda$, and $\bm{B}$ be a matrix such that $\mathop\mathrm{Tr}(\bm{B})\leq\mathbb{1}$. Let $\bm{p}$ and $\bm{g}$ be vectors, $\bm{q}$ and $\bm{h}$ be regular vectors, and $r$ be a scalar such that $\bm{h}^{-}\bm{B}^{\ast}\bm{g}\leq\mathbb{1}$ and $\lambda\oplus(\bm{q}^{-}\bm{p})^{1/2}\oplus r>\mathbb{0}$.

Then, the minimum value in problem \eqref{P-xAxxpqxr-Bxgx-xh} is equal to
\begin{multline*}
\theta
=
\bigoplus_{k=1}^{n}\mathop\mathrm{tr}\nolimits^{1/k}(\bm{S}_{k})
\oplus
\bigoplus_{k=1}^{n-1}(\bm{h}^{-}\bm{T}_{k}\bm{g})^{1/k}
\\
\oplus
\bigoplus_{k=0}^{n-1}(\bm{q}^{-}\bm{T}_{k}\bm{g}\oplus\bm{h}^{-}\bm{T}_{k}\bm{p})^{1/(k+1)}
\oplus
\bigoplus_{k=0}^{n-1}(\bm{q}^{-}\bm{T}_{k}\bm{p})^{1/(k+2)}
\oplus
r,
\end{multline*}
and all regular solutions are given by
\begin{equation*}
\bm{x}
=
(\theta^{-1}\bm{A}\oplus\bm{B})^{\ast}\bm{u},
%\label{E-x-theta1ABu}
\end{equation*}
where $\bm{u}$ is any regular vector that satisfies the conditions
\begin{equation*}
\theta^{-1}\bm{p}\oplus\bm{g}
\leq
\bm{u}
\leq
((\theta^{-1}\bm{q}^{-}\oplus\bm{h}^{-})(\theta^{-1}\bm{A}\oplus\bm{B})^{\ast})^{-}.
%\label{I-theta1pg-u-theta1qhtheta1AB}
\end{equation*}
\end{theorem}
\begin{proof}
We introduce a parameter to represent the minimum value of the objective function, and then reduce the problem to solving a parametrized system of linear inequalities. The necessary and sufficient conditions for the system to have regular solutions serve to evaluate the parameter, whereas the general solution of the system is taken as a complete solution to the initial optimization problem.

Denote by $\theta$ the minimum of the objective function over all regular vectors $\bm{x}$. Then, all regular solutions to problem \eqref{P-xAxxpqxr-Bxgx-xh} are determined by the system
\begin{equation}
\begin{aligned}
\bm{x}^{-}\bm{A}\bm{x}\oplus\bm{x}^{-}\bm{p}\oplus\bm{q}^{-}\bm{x}\oplus r
&\leq
\theta,
\\
\bm{B}\bm{x}\oplus\bm{g}
&\leq
\bm{x},
\\
\bm{x}
&\leq
\bm{h}.
\end{aligned}
\label{I-xAxxpqxrtheta-Bxgx-xh}
\end{equation}

The first inequality at \eqref{I-xAxxpqxrtheta-Bxgx-xh} is equivalent to the four inequalities
\begin{equation}
\bm{x}^{-}\bm{A}\bm{x}
\leq
\theta,
\qquad
\bm{x}^{-}\bm{p}
\leq
\theta,
\qquad
\bm{q}^{-}\bm{x}
\leq
\theta,
\qquad
r
\leq
\theta.
\label{I-xAx-xp-qx-r-theta}
\end{equation}

We use these inequalities to derive a lower bound for $\theta$ and verify that $\theta\ne\mathbb{0}$. The first inequality at \eqref{I-xAx-xp-qx-r-theta} and Lemma~\ref{L-minxAx} imply that $\theta\geq\bm{x}^{-}\bm{A}\bm{x}\geq\lambda$. From the next two inequalities and a property of the conjugate transposition, we derive $\theta^{2}\geq\bm{q}^{-}\bm{x}\bm{x}^{-}\bm{p}\geq\bm{q}^{-}\bm{p}$, which gives $\theta\geq(\bm{q}^{-}\bm{p})^{1/2}$. Since $\theta\geq r$ as well, we finally obtain a lower bound for $\theta$ in the form
\begin{equation}
\theta
\geq
\lambda\oplus(\bm{q}^{-}\bm{p})^{1/2}\oplus r,
\label{I-thetalambdaqpr}
\end{equation}
where the right-hand side is nonzero by the conditions of the theorem. 

We can now multiply the first two inequalities at \eqref{I-xAx-xp-qx-r-theta} by $\theta^{-1}$, and then apply Lemma~\ref{I-Axd} to the first three. As a result, we have the inequalities
$$
\theta^{-1}\bm{A}\bm{x}
\leq
\bm{x},
\qquad
\theta^{-1}\bm{p}
\leq
\bm{x},
\qquad
\bm{x}
\leq
\theta\bm{q}.
$$

As the next step, we combine these inequalities with those in the system at \eqref{I-xAxxpqxrtheta-Bxgx-xh}. Specifically, the first two inequalities together with $\bm{B}\bm{x}\oplus\bm{g}\leq\bm{x}$ give the inequality $(\theta^{-1}\bm{A}\oplus\bm{B})\bm{x}\oplus\theta^{-1}\bm{p}\oplus\bm{g}\leq\bm{x}$.

In addition, we take the inequalities $\bm{x}\leq\theta\bm{q}$ and $\bm{x}\leq\bm{h}$, and put them into the forms $\bm{x}^{-}\geq\theta^{-1}\bm{q}^{-}$ and $\bm{x}^{-}\geq\bm{h}^{-}$. The last two inequalities are combined into one, which is then rewritten to give $\bm{x}\leq(\theta^{-1}\bm{q}^{-}\oplus\bm{h}^{-})^{-}$.

By coupling the obtained inequalities, we represent system \eqref{I-xAxxpqxrtheta-Bxgx-xh} as
\begin{equation}
\begin{aligned}
(\theta^{-1}\bm{A}\oplus\bm{B})\bm{x}
\oplus
\theta^{-1}\bm{p}\oplus\bm{g}
&\leq
\bm{x},
\\
\bm{x}
&\leq
(\theta^{-1}\bm{q}^{-}\oplus\bm{h}^{-})^{-}.
\end{aligned}
\label{I-thetaABxthetapgx-xthetaqh}
\end{equation}

Considering that system \eqref{I-thetaABxthetapgx-xthetaqh} has the form of \eqref{I-Axbx-xd}, we can apply Lemma~\ref{L-Axbx-xd} to examine this system. By the lemma, the necessary and sufficient condition for \eqref{I-thetaABxthetapgx-xthetaqh} to have regular solutions takes the form
\begin{equation*}
\mathop\mathrm{Tr}(\theta^{-1}\bm{A}\oplus\bm{B})
\oplus
(\theta^{-1}\bm{q}^{-}\oplus\bm{h}^{-})(\theta^{-1}\bm{A}\oplus\bm{B})^{\ast}(\theta^{-1}\bm{p}\oplus\bm{g})
\leq
\mathbb{1}.
%\label{I-thetaAB1-thetaqhthetaABthetapg1}
\end{equation*}

To solve this inequality with respect to the parameter $\theta$, we put it in a more convenient form by expanding the left-hand side in powers of $\theta$.

As a starting point, we examine the matrix asterate
$$
(\theta^{-1}\bm{A}\oplus\bm{B})^{\ast}
=
\bigoplus_{k=0}^{n-1}(\theta^{-1}\bm{A}\oplus\bm{B})^{k}
=
\bm{I}
\oplus
\bigoplus_{k=1}^{n-1}(\theta^{-1}\bm{A}\oplus\bm{B})^{k}.
$$

After application of \eqref{E-ABk} to the second term, we rearrange the expression to collect terms with the same power of $\theta$, and then use \eqref{E-Tk} to write
\begin{multline*}
(\theta^{-1}\bm{A}\oplus\bm{B})^{\ast}
=
\bigoplus_{k=1}^{n-1}\mathop{\bigoplus\hspace{1.0em}}_{0\leq i_{0}+i_{1}+\cdots+i_{k}\leq n-k-1}\theta^{-k}\bm{B}^{i_{0}}(\bm{A}\bm{B}^{i_{1}}\cdots\bm{A}\bm{B}^{i_{k}})
\oplus
\bigoplus_{k=0}^{n-1}\bm{B}^{k}
\\
=
\bigoplus_{k=1}^{n-1}\theta^{-k}\bm{T}_{k}
\oplus
\bm{T}_{0}
=
\bigoplus_{k=0}^{n-1}\theta^{-k}\bm{T}_{k}.
\end{multline*}

By using \eqref{E-trABk}, \eqref{E-Sk}, \eqref{E-Tk} and properties of the trace function, we also have
\begin{multline*}
\mathop\mathrm{Tr}(\theta^{-1}\bm{A}\oplus\bm{B})
=
\bigoplus_{k=1}^{n}\mathop\mathrm{tr}(\theta^{-1}\bm{A}\oplus\bm{B})^{k}
\\
=
\bigoplus_{k=1}^{n}\mathop{\bigoplus\hspace{1.2em}}_{0\leq i_{1}+\cdots+i_{k}\leq n-k}\theta^{-k}\mathop\mathrm{tr}(\bm{A}\bm{B}^{i_{1}}\cdots\bm{A}\bm{B}^{i_{k}})
\oplus
\bigoplus_{k=1}^{n}\mathop\mathrm{tr}(\bm{B}^{k})
\\
=
\bigoplus_{k=1}^{n}\theta^{-k}\mathop\mathrm{tr}(\bm{S}_{k})
\oplus
\mathop\mathrm{Tr}(\bm{B}).
\end{multline*}

Substitution of these results into the condition for regular solutions yields
$$
\bigoplus_{k=1}^{n}\theta^{-k}\mathop\mathrm{tr}(\bm{S}_{k})
\oplus
\bigoplus_{k=0}^{n-1}\theta^{-k}(\theta^{-1}\bm{q}^{-}\oplus\bm{h}^{-})\bm{T}_{k}(\theta^{-1}\bm{p}\oplus\bm{g})
\oplus
\mathop\mathrm{Tr}(\bm{B})
\leq
\mathbb{1}.
$$

Since $\mathop\mathrm{Tr}(\bm{B})\leq\mathbb{1}$ by the conditions of the theorem, the term $\mathop\mathrm{Tr}(\bm{B})$ does not affect the solution of the inequality, and hence can be omitted. The remaining inequality is equivalent to the system of inequalities
\begin{align*}
\theta^{-k}\mathop\mathrm{tr}(\bm{S}_{k})
&\leq
\mathbb{1},
\qquad
k=1,\ldots,n;
\\
\theta^{-k}(\theta^{-1}\bm{q}^{-}\oplus\bm{h}^{-})\bm{T}_{k}(\theta^{-1}\bm{p}\oplus\bm{g})
&\leq
\mathbb{1},
\qquad
k=0,1,\ldots,n-1;
\end{align*}
which can be further split into the system
\begin{align*}
\theta^{-k}\mathop\mathrm{tr}(\bm{S}_{k})
&\leq
\mathbb{1},
\qquad
k=1,\ldots,n;
\\
\theta^{-k}\bm{h}^{-}\bm{T}_{k}\bm{g}
&\leq
\mathbb{1},
\\
\theta^{-k-1}(\bm{q}^{-}\bm{T}_{k}\bm{g}\oplus\bm{h}^{-}\bm{T}_{k}\bm{p})
&\leq
\mathbb{1},
\\
\theta^{-k-2}\bm{q}^{-}\bm{T}_{k}\bm{p}
&\leq
\mathbb{1},
\qquad
k=0,1,\ldots,n-1.
\end{align*}

Note that $\bm{h}^{-}\bm{T}_{0}\bm{g}=\bm{h}^{-}\bm{B}^{\ast}\bm{g}\leq\mathbb{1}$ by the conditions of the theorem, and thus the second inequality in the system is valid at $k=0$ for all $\theta>\mathbb{0}$.

By solving the inequalities, we have
\begin{align*}
\theta
&\geq
\mathop\mathrm{tr}\nolimits^{1/k}(\bm{S}_{k}),
&&
k=1,\ldots,n;
\\
\theta
&\geq
(\bm{h}^{-}\bm{T}_{k}\bm{g})^{1/k},
&&
k=1,\ldots,n-1;
\\
\theta
&\geq
(\bm{q}^{-}\bm{T}_{k}\bm{g}\oplus\bm{h}^{-}\bm{T}_{k}\bm{p})^{1/(k+1)},
\\
\theta
&\geq
(\bm{q}^{-}\bm{T}_{k}\bm{p})^{1/(k+2)},
&&
k=0,1,\ldots,n-1.
\end{align*}

The obtained solutions can be combined into one equivalent inequality
\begin{multline*}
\theta
\geq
\bigoplus_{k=1}^{n}\mathop\mathrm{tr}\nolimits^{1/k}(\bm{S}_{k})
\oplus
\bigoplus_{k=1}^{n-1}(\bm{h}^{-}\bm{T}_{k}\bm{g})^{1/k}
\\
\oplus
\bigoplus_{k=0}^{n-1}(\bm{q}^{-}\bm{T}_{k}\bm{g}\oplus\bm{h}^{-}\bm{T}_{k}\bm{p})^{1/(k+1)}
\oplus
\bigoplus_{k=0}^{n-1}(\bm{q}^{-}\bm{T}_{k}\bm{p})^{1/(k+2)}.
\end{multline*}

We have to couple the lower bound given by \eqref{I-thetalambdaqpr} with that defined by the last inequality. It is not difficult to verify that the right-hand side of this inequality already takes account of the terms $\lambda$ and $(\bm{q}^{-}\bm{p})^{1/2}$ presented in \eqref{I-thetalambdaqpr}. Indeed, considering that $\bm{S}_{k}\geq\bm{A}^{k}$, we have
$$
\bigoplus_{k=1}^{n}\mathop\mathrm{tr}\nolimits^{1/k}(\bm{S}_{k})
\geq
\bigoplus_{k=1}^{n}\mathop\mathrm{tr}\nolimits^{1/k}(\bm{A}^{k})
=
\lambda.
$$

Moreover, since $\bm{T}_{0}=\bm{B}^{\ast}\geq\bm{I}$, it is easy to see that
$$
\bigoplus_{k=0}^{n-1}(\bm{q}^{-}\bm{T}_{k}\bm{p})^{1/(k+2)}
\geq
(\bm{q}^{-}\bm{T}_{0}\bm{p})^{1/2}
\geq
(\bm{q}^{-}\bm{p})^{1/2}.
$$

By combining all lower bounds obtained for $\theta$, we arrive at the inequality
\begin{multline*}
\theta
\geq
\bigoplus_{k=1}^{n}\mathop\mathrm{tr}\nolimits^{1/k}(\bm{S}_{k})
\oplus
\bigoplus_{k=1}^{n-1}(\bm{h}^{-}\bm{T}_{k}\bm{g})^{1/k}
\\
\oplus
\bigoplus_{k=0}^{n-1}(\bm{q}^{-}\bm{T}_{k}\bm{g}\oplus\bm{h}^{-}\bm{T}_{k}\bm{p})^{1/(k+1)}
\oplus
\bigoplus_{k=0}^{n-1}(\bm{q}^{-}\bm{T}_{k}\bm{p})^{1/(k+2)}
\oplus
r.
\end{multline*}

Since $\theta$ is assumed to be the minimal value of the objective function, this inequality must hold as an equality, which yields the desired minimum.

Finally, we take the minimum value of $\theta$, and then apply Lemma~\ref{L-Axbx-xd} to obtain all solutions of the system at \eqref{I-thetaABxthetapgx-xthetaqh} in the form
$$
\bm{x}
=
(\theta^{-1}\bm{A}\oplus\bm{B})^{\ast}\bm{u},
\qquad
\theta^{-1}\bm{p}\oplus\bm{g}
\leq
\bm{u}
\leq((\theta^{-1}\bm{q}^{-}\oplus\bm{h}^{-})(\theta^{-1}\bm{A}\oplus\bm{B})^{\ast})^{-}.
$$

Because the solution obtained is also a complete solution of the initial optimization problem, this ends the proof of the theorem.
\qed
\end{proof}

We conclude this section with a brief discussion of the computational complexity of the solution obtained to see that it is polynomial in the dimension $n$. Indeed, this complexity is determined by the complexity of computing the minimum value $\theta$, as the other components of the solution are given by a finite number of matrix and vector operations, and thus obviously take no more than polynomial time.

Furthermore, it directly follows from the expression for $\theta$ that, if the evaluation of the matrix sequences $\bm{S}_{1},\ldots,\bm{S}_{n}$ and $\bm{T}_{0},\ldots,\bm{T}_{n-1}$ has polynomial complexity, then so has that of $\theta$. Considering that $\bm{S}_{k+1}=\bm{A}\bm{T}_{k}$ for all $k=0,\ldots,n-1$, we need to verify that $\bm{T}_{k}$ can be obtained in polynomial time.

To describe a polynomial scheme of calculating $\bm{T}_{k}$, we first write
$$
\bm{T}_{k}
=
\bigoplus_{l=1}^{n-k-1}\bm{Q}_{kl},
\qquad
\bm{Q}_{kl}
=
\mathop{\bigoplus\hspace{-0.4em}}_{i_{0}+i_{1}+\cdots+i_{k}=l}
\bm{B}^{i_{0}}(\bm{A}\bm{B}^{i_{1}}\cdots\bm{A}\bm{B}^{i_{k}}),
$$
where $\bm{Q}_{kl}$ is the sum of all matrix products that are comprised of $k$ factors equal to $\bm{A}$ and $l$ factors equal to $\bm{B}$, with $\bm{Q}_{k0}=\bm{A}^{k}$, $\bm{Q}_{0l}=\bm{B}^{l}$ and $\bm{Q}_{00}=\bm{I}$. In this case, the evaluation of the matrices $\bm{T}_{0},\ldots,\bm{T}_{n-1}$ reduces to computing the matrices $\bm{Q}_{kl}$ for all $k=0,\ldots,n-1$ and $l=0,\ldots,n-k-1$.

Furthermore, we note that the recurrent relation $\bm{Q}_{kl}=\bm{A}\bm{Q}_{k-1,l}\oplus\bm{B}\bm{Q}_{k,l-1}$ holds for all $k,l=1,2,\ldots$ It is clear that this relation offers a natural way to obtain successively all matrices $\bm{Q}_{kl}$, using two matrix multiplications and one matrix addition per matrix. Since the overall number of matrices involved in computation is $1+2+\cdots+(n-1)=n(n-1)/2$, the computation of all matrices $\bm{T}_{k}$ requires polynomial time, and thus the entire solution has polynomial complexity.

\subsection{Some special cases}

As direct consequences of the result obtained, we now find solutions to special cases of problems \eqref{P-xAxxpqxr-Bxgx-xh} and \eqref{P-xAxxpqxr-Bxx-gxh} with reduced sets of constraints. To begin with, eliminate the first constraint in \eqref{P-xAxxpqxr-Bxx-gxh} and consider the problem 
\begin{equation}
\begin{aligned}
&
\text{minimize}
&&
\bm{x}^{-}\bm{A}\bm{x}\oplus\bm{x}^{-}\bm{p}\oplus\bm{q}^{-}\bm{x}\oplus r,
\\
&
\text{subject to}
&&
\bm{g}
\leq
\bm{x}
\leq
\bm{h}.
\end{aligned}
\label{P-xAxxpqxr-gxh}
\end{equation}

Clearly, the solution to this problem can be derived from that of \eqref{P-xAxxpqxr-Bxx-gxh} by setting $\bm{B}=\mathbb{0}$. Under this condition, we have $\bm{S}_{k}=\bm{A}^{k}$ and $\bm{T}_{k}=\bm{A}^{k}$, whereas the solution is described as follows.
\begin{corollary}
\label{C-xAxxpqxr-gxh}
Let $\bm{A}$ be a matrix with spectral radius $\lambda$. Let $\bm{p}$ and $\bm{g}$ be vectors, $\bm{q}$ and $\bm{h}$ be regular vectors, and $r$ be a scalar such that $\bm{h}^{-}\bm{g}\leq\mathbb{1}$ and $\lambda\oplus(\bm{q}^{-}\bm{p})^{1/2}\oplus r>\mathbb{0}$. Then, the minimum value in problem \eqref{P-xAxxpqxr-gxh} is equal to
\begin{multline*}
\theta
=
\lambda
\oplus
\bigoplus_{k=1}^{n-1}(\bm{h}^{-}\bm{A}^{k}\bm{g})^{1/k}
\oplus
\bigoplus_{k=0}^{n-1}(\bm{q}^{-}\bm{A}^{k}\bm{g}\oplus\bm{h}^{-}\bm{A}^{k}\bm{p})^{1/(k+1)}
\\
\oplus
\bigoplus_{k=0}^{n-1}(\bm{q}^{-}\bm{A}^{k}\bm{p})^{1/(k+2)}
\oplus
r,
\end{multline*}
and all regular solutions are given by
\begin{equation*}
\bm{x}
=
(\theta^{-1}\bm{A})^{\ast}\bm{u},
\qquad
\theta^{-1}\bm{p}\oplus\bm{g}
\leq
\bm{u}
\leq
((\theta^{-1}\bm{q}^{-}\oplus\bm{h}^{-})(\theta^{-1}\bm{A})^{\ast})^{-}.
\end{equation*}
\end{corollary}

Furthermore, we consider another special case of \eqref{P-xAxxpqxr-Bxx-gxh}, which takes the form
\begin{equation}
\begin{aligned}
&
\text{minimize}
&&
\bm{x}^{-}\bm{A}\bm{x}\oplus\bm{x}^{-}\bm{p}\oplus\bm{q}^{-}\bm{x}\oplus r,
\\
&
\text{subject to}
&&
\bm{B}\bm{x}
\leq
\bm{x}.
\end{aligned}
\label{P-xAxxpqxr-Bxx}
\end{equation}

After slight modification of the proof of Theorem~\ref{T-xAxxpqxr-Bxgx-xh}, we arrive at the next result, which can also be obtained directly by putting $\bm{g}=\bm{0}$ and $\bm{h}^{-}=\bm{0}^{T}$ in the solution of problem \eqref{P-xAxxpqxr-Bxx-gxh}.
\begin{corollary}
\label{C-xAxxpqxr-Bxx}
Let $\bm{A}$ be a matrix with spectral radius $\lambda$, and $\bm{B}$ be a matrix such that $\mathop\mathrm{Tr}(\bm{B})\leq\mathbb{1}$. Let $\bm{p}$ be a vector, $\bm{q}$ be a regular vector, and $r$ be a scalar such that $\lambda\oplus(\bm{q}^{-}\bm{p})^{1/2}\oplus r>\mathbb{0}$. Then, the minimum value in problem \eqref{P-xAxxpqxr-Bxx} is equal to
$$
\theta
=
\bigoplus_{k=1}^{n}\mathop\mathrm{tr}\nolimits^{1/k}(\bm{S}_{k})
\oplus
\bigoplus_{k=0}^{n-1}(\bm{q}^{-}\bm{T}_{k}\bm{p})^{1/(k+2)}
\oplus
r,
$$
and all regular solutions are given by
\begin{equation*}
\bm{x}
=
(\theta^{-1}\bm{A}\oplus\bm{B})^{\ast}\bm{u},
\qquad
\theta^{-1}\bm{p}
\leq
\bm{u}
\leq
\theta(\bm{q}^{-}(\theta^{-1}\bm{A}\oplus\bm{B})^{\ast})^{-}.
%\label{I-theta1pg-u-theta1qhtheta1AB}
\end{equation*}
\end{corollary}

Finally, note that eliminating both inequality constraints in \eqref{P-xAxxpqxr-Bxgx-xh} leads to the same solution as that provided by Theorem~\ref{T-xAxxpqxr}.

\section{Application to project scheduling}
\label{S-APS}

We now apply the result obtained to solve an example problem, which is drawn from project scheduling \cite{Demeulemeester2002Project,Neumann2003Project,Tkindt2006Multicriteria} and serves to motivate and illustrate the study. 
 
Consider a project consisting of a set of activities that are performed in parallel under various temporal constraints given by precedence relationships, release times and time windows. The precedence relationships are defined for each pair of activities and include the start-finish constraints on the minimum allowed time lag between the initiation of one activity and completion of another, and the start-start constraints on the minimum lag between the initiations of the activities. Once an activity starts, it continues to its completion, and no interruption is allowed. The activities are completed as soon as possible under the start-finish constraints.

The release time constraints take the form of release dates and release deadlines to specify that the activities cannot be initiated, respectively, before and after prescribed times. The time windows are given by lower and upper boundaries, and determine the minimum time slots preallocated to each activity. The activities have to occupy their time windows entirely. If the initiation time of an activity falls to the right of the lower boundary of its window, this time is adjusted by shifting to this boundary. In a similar way, the completion time is set to the upper boundary if it appears to the left of this boundary.

Each activity in the project has its flow-time defined as the duration of the interval between the adjusted initiation and completion times. A schedule is optimal if it minimizes the maximum flow-times over all activities. The problem of interest is to find the initiation and completion times of the activities to provide an optimal schedule subject to the temporal constraints described above.

\subsection{Representation and solution of scheduling problem}

Suppose a project involves $n$ activities. For each activity $i=1,\ldots,n$, let $x_{i}$ be the initiation and $y_{i}$ the completion time. We denote the minimum possible time lags between the initiation of activity $j=1,\ldots,n$ and the completion of $i$ by $a_{ij}$, and between the initiations of $j$ and $i$ by $b_{ij}$. If a time lag is not specified for a pair of activities, we set it to $-\infty$.

The start-finish constraints yield the equalities
$$
y_{i}
=
\max(a_{i1}+x_{1},\ldots,a_{in}+x_{n}),
\qquad
i=1,\ldots,n;
$$
whereas the start-start constraints lead to the inequalities
$$
x_{i}
\geq
\max(b_{i1}+x_{1},\ldots,b_{in}+x_{n}),
\qquad
i=1,\ldots,n.
$$

Let $g_{i}$ and $h_{i}$ be, respectively, the possible earliest and latest initiation times. The release date and release deadline constraints are given by the inequalities
$$
g_{i}
\leq
x_{i}
\leq
h_{i},
\qquad
i=1,\ldots,n.
$$

Then, we denote the lower and upper boundaries of the minimum time window for activity $i$ by $q_{i}$ and $p_{i}$, respectively. Let $s_{i}$ be the adjusted initiation time and $t_{i}$ the adjusted completion time of the activity. Since the time window must be fully occupied, we have
$$
s_{i}
=
\min(x_{i},q_{i})
=
-
\max(-x_{i},-q_{i}),
\quad
t_{i}
=
\max(y_{i},p_{i}),
\qquad
i=1,\ldots,n.
$$

Finally, the maximum flow-time over all activities is given by
$$
\max(t_{1}-s_{1},\ldots,t_{n}-s_{n}).
$$

We are now in a position to represent the optimal scheduling problem of interest as that of finding $x_{i}$, $y_{i}$, $s_{i}$ and $t_{i}$ for all $i=1,\ldots,n$ to
\begin{equation*}
\begin{aligned}
&
\text{minimize}
&&
\max_{1\leq i\leq n}(t_{i}-s_{i}),
\\
&
\text{subject to}
&&
s_{i}
=
-
\max(-x_{i},-q_{i}),
&&
t_{i}
=
\max(y_{i},p_{i}),
\\
&&&
y_{i}
=
\max_{1\leq j\leq n}(a_{ij}+x_{j}),
&&
x_{i}
\geq
\max_{1\leq j\leq n}(b_{ij}+x_{j}),
\\
&&&
g_{i}
\leq
x_{i}
\leq
h_{i},
&&
i=1,\ldots,n.
\end{aligned}
\end{equation*}

It is not difficult to see that this problem can be represented and solved within the framework of linear programming, which generally offers algorithmic solutions rather than a direct complete solution in an explicit form.

To obtain a direct solution, we place the problem in the context of tropical mathematics. Considering that the problem is formulated only in terms of the operations of maximum, ordinary addition, and additive inversion, we can rewrite it in the setting of the semifield $\mathbb{R}_{\max,+}$ as follows: 
\begin{equation*}
\begin{aligned}
&
\text{minimize}
&&
\bigoplus_{i=1}^{n}s_{i}^{-1}t_{i},
\\
&
\text{subject to}
&&
s_{i}
=
(x_{i}^{-1}\oplus q_{i}^{-1})^{-1},
&&
t_{i}
=
y_{i}\oplus p_{i},
\\
&&&
y_{i}
=
\bigoplus_{j=1}^{n}a_{ij}x_{j},
&&
x_{i}
\geq
\bigoplus_{j=1}^{n}b_{ij}x_{j},
\\
&&&
g_{i}
\leq
x_{i}
\leq
h_{i},
&&
i=1,\ldots,n.
\end{aligned}
\end{equation*}

Furthermore, we put the problem into a compact vector form. We introduce the matrix-vector notation
$$
\bm{A}
=
(a_{ij}),
\quad
\bm{B}
=
(b_{ij}),
\quad
\bm{x}
=
(x_{i}),
\quad
\bm{y}
=
(y_{i}),
\quad
\bm{g}
=
(g_{i}),
\quad
\bm{h}
=
(h_{i}),
$$
and write the start-finish, start-start and release time constraints as
$$
\bm{y}
=
\bm{A}\bm{x},
\qquad
\bm{x}
\geq
\bm{B}\bm{x},
\qquad
\bm{g}
\leq
\bm{x}
\leq
\bm{h}.
$$

To take into account the time window boundaries and adjusted times, we use the vector notation
$$
\bm{s}
=
(s_{i}),
\qquad
\bm{t}
=
(t_{i}),
\qquad
\bm{p}
=
(p_{i}),
\qquad
\bm{q}
=
(q_{i}).
$$

The vectors of adjusted initiation and completion times take the form
$$
\bm{s}
=
(\bm{x}^{-}\oplus\bm{q}^{-})^{-},
\qquad
\bm{t}
=
\bm{y}\oplus\bm{p}.
$$

The optimal scheduling problem to minimize the maximum flow-time subject to the temporal constraints under consideration now becomes
\begin{equation}
\begin{aligned}
&
\text{minimize}
&&
\bm{s}^{-}\bm{t},
\\
&
\text{subject to}
&&
\bm{s}^{-}
=
\bm{x}^{-}\oplus\bm{q}^{-},
\quad
\bm{t}
=
\bm{y}\oplus\bm{p},
\\
&&&
\bm{A}\bm{x}
=
\bm{y},
\quad
\bm{B}\bm{x}
\leq
\bm{x},
\\
&&&
\bm{g}
\leq
\bm{x}
\leq
\bm{h}.
\end{aligned}
\label{P-st-sxq-typ-Axy-Bxx-gxh}
\end{equation}

Note that, in the context of scheduling problems, it is natural to consider the matrix $\bm{A}$ as column-regular matrix, and the vectors $\bm{p}$, $\bm{q}$ and $\bm{h}$ as regular.

A complete solution to the problem is given by the next result.
\begin{theorem}
\label{T-st-sxq-typ-Axy-Bxx-gxh}
Let $\bm{A}$ be a column-regular matrix, and $\bm{B}$ be a matrix such that $\mathop\mathrm{Tr}(\bm{B})\leq\mathbb{1}$. Let $\bm{p}$, $\bm{q}$ and $\bm{h}$ be regular vectors and $\bm{g}$ be a vector such that $\bm{h}^{-}\bm{B}^{\ast}\bm{g}\leq\mathbb{1}$. Then, the minimum flow-time in problem \eqref{P-st-sxq-typ-Axy-Bxx-gxh} is equal to
\begin{multline}
\theta
=
\bigoplus_{k=1}^{n}\mathop\mathrm{tr}\nolimits^{1/k}(\bm{S}_{k})
\oplus
\bigoplus_{k=1}^{n-1}(\bm{h}^{-}\bm{T}_{k}\bm{g})^{1/k}
\oplus
\bigoplus_{k=1}^{n}(\bm{q}^{-}\bm{S}_{k}\bm{g})^{1/k}
\\
\oplus
\bigoplus_{k=0}^{n-1}(\bm{h}^{-}\bm{T}_{k}\bm{p})^{1/(k+1)}
\oplus
\bigoplus_{k=0}^{n}(\bm{q}^{-}\bm{S}_{k}\bm{p})^{1/(k+1)},
\label{E-theta}
\end{multline}
and the vectors of initiation and completion times are given by
\begin{align}
\bm{x}
&=
(\theta^{-1}\bm{A}\oplus\bm{B})^{\ast}\bm{u},
&
\bm{y}
&=
\bm{A}(\theta^{-1}\bm{A}\oplus\bm{B})^{\ast}\bm{u},
\label{E-x-y}
\\
\bm{s}
&=
(((\theta^{-1}\bm{A}\oplus\bm{B})^{\ast}\bm{u})^{-}\oplus\bm{q}^{-})^{-},
&
\bm{t}
&=
\bm{A}(\theta^{-1}\bm{A}\oplus\bm{B})^{\ast}\bm{u}\oplus\bm{p},
\label{E-s-t}
\end{align}
where $\bm{u}$ is any vector that satisfies the conditions
\begin{equation}
\theta^{-1}\bm{p}\oplus\bm{g}
\leq
\bm{u}
\leq
((\theta^{-1}\bm{q}^{-}\bm{A}\oplus\bm{h}^{-})(\theta^{-1}\bm{A}\oplus\bm{B})^{\ast})^{-}.
\label{I-thetapg-u-thetaqAhthetaAB}
\end{equation}
\end{theorem}
\begin{proof}
First, we eliminate the vectors $\bm{s}$ and $\bm{t}$ from problem \eqref{P-st-sxq-typ-Axy-Bxx-gxh} by representing the objective function as
$$
\bm{s}^{-}\bm{t}
=
(\bm{x}^{-}\oplus\bm{q}^{-})(\bm{y}\oplus\bm{p})
=
\bm{x}^{-}\bm{y}\oplus\bm{q}^{-}\bm{y}\oplus\bm{x}^{-}\bm{p}\oplus\bm{q}^{-}\bm{p}.
$$

Furthermore, we substitute $\bm{y}=\bm{A}\bm{x}$ to reduce \eqref{P-st-sxq-typ-Axy-Bxx-gxh} to the problem
\begin{equation*}
\begin{aligned}
&
\text{minimize}
&&
\bm{x}^{-}\bm{A}\bm{x}\oplus\bm{q}^{-}\bm{A}\bm{x}\oplus\bm{x}^{-}\bm{p}\oplus\bm{q}^{-}\bm{p},
\\
&
\text{subject to}
&&
\bm{B}\bm{x}
\leq
\bm{x},
\\
&&&
\bm{g}
\leq
\bm{x}
\leq
\bm{h},
\end{aligned}
\end{equation*}
which has the form of \eqref{P-xAxxpqxr-Bxx-gxh}, where $\bm{q}^{-}$ is replaced by $\bm{q}^{-}\bm{A}$ and $r$ by $\bm{q}^{-}\bm{p}$.

To apply Theorem~\ref{T-xAxxpqxr-Bxgx-xh}, we note that, under the given conditions, the conditions of the theorem are satisfied as well. Specifically, since both vectors $\bm{p}$ and $\bm{q}$ are regular, we have $r=\bm{q}^{-}\bm{p}>\mathbb{0}$, and thus provide the last condition of Theorem~\ref{T-xAxxpqxr-Bxgx-xh}.

Next, we refine the expression for $\theta$ by applying the identity $\bm{A}\bm{T}_{k}=\bm{S}_{k+1}$, which is valid for all $k=0,\ldots,n-1$.

After some rearrangement of sums, we arrive at \eqref{E-theta}. Both the representation for $\bm{x}$ at \eqref{E-x-y} and the condition on $\bm{u}$ at \eqref{I-thetapg-u-thetaqAhthetaAB} are directly obtained from Theorem~\ref{T-xAxxpqxr-Bxgx-xh}. The other expressions in \eqref{E-x-y} and \eqref{E-s-t} are immediate consequences.
\qed
\end{proof}

As before, the solutions to special cases without constraints are readily derived from the general solution offered by Theorem~\ref{T-st-sxq-typ-Axy-Bxx-gxh}. Specifically, we eliminate the boundary constraint $\bm{g}\leq\bm{x}\leq\bm{h}$ by setting $\bm{g}=\bm{0}$ and $\bm{h}^{-}=\bm{0}^{T}$, and/or the linear inequality constraint with matrix in the form $\bm{B}\bm{x}\leq\bm{x}$ by setting $\bm{B}=\bm{0}$, which further yields the substitutions $\bm{S}_{k}=\bm{A}^{k}$ and $\bm{T}_{k}=\bm{A}^{k}$.

\subsection{Numerical example}

To provide a clear illustration of the above result and of the computational technique, we solve in detail a simple low-dimensional problem. Even though the example under consideration is somewhat artificial, it well demonstrates the applicability of the solution to real-world problems of higher dimension. 

Let us examine a project that involves $n=3$ activities under constraints given by the matrices
$$
\bm{A}
=
\left(
\begin{array}{ccr}
4 & 0 & -\infty
\\
2 & 3 & 1
\\
1 & 1 & 3
\end{array}
\right),
\qquad
\bm{B}
=
\left(
\begin{array}{rrr}
-\infty & -1 & 1
\\
0 & -\infty & 2
\\
-1 & -\infty & -\infty
\end{array}
\right),
$$
and by the vectors
$$
\bm{p}
=
\left(
\begin{array}{c}
4
\\
4
\\
5
\end{array}
\right),
\qquad
\bm{q}
=
\left(
\begin{array}{c}
3
\\
2
\\
1
\end{array}
\right),
\qquad
\bm{g}
=
\left(
\begin{array}{c}
0
\\
0
\\
1
\end{array}
\right),
\qquad
\bm{h}
=
\left(
\begin{array}{c}
2
\\
3
\\
3
\end{array}
\right).
$$

We start with the verification of the existence conditions for regular solutions in Theorem~\ref{T-st-sxq-typ-Axy-Bxx-gxh}. First note that the matrix $\bm{A}$ is obviously column-regular. In what follows, we need the powers of the matrix $\bm{A}$, which have the form
$$
\bm{A}^{2}
=
\left(
\begin{array}{ccc}
8 & 4 & 1
\\
6 & 6 & 4
\\
5 & 4 & 6
\end{array}
\right),
\qquad
\bm{A}^{3}
=
\left(
\begin{array}{ccc}
12 & 8 & 5
\\
10 & 9 & 7
\\
9 & 7 & 9
\end{array}
\right).
$$

Then, we take the matrix $\bm{B}$ and calculate
$$
\bm{B}^{2}
=
\left(
\begin{array}{rrc}
0 & -\infty & 1
\\
1 & -1 & 1
\\
-\infty & -2 & 0
\end{array}
\right),
\qquad
\bm{B}^{3}
=
\left(
\begin{array}{rrc}
0 & -1 & 1
\\
0 & 0 & 2
\\
-1 & -\infty & 0
\end{array}
\right),
\qquad
\mathop\mathrm{Tr}(\bm{B})
=
0.
$$

Furthermore, we successively obtain 
$$
\bm{B}^{\ast}
=
\left(
\begin{array}{rrc}
0 & -1 & 1
\\
1 & 0 & 2
\\
-1 & -2 & 0
\end{array}
\right),
\qquad
\bm{h}^{-}\bm{B}^{\ast}
=
\left(
\begin{array}{rrr}
-2 & -3 & -1
\end{array}
\right),
\qquad
\bm{h}^{-}\bm{B}^{\ast}\bm{g}
=
0.
$$

Since $\mathop\mathrm{Tr}(\bm{B})=\bm{h}^{-}\bm{B}^{\ast}\bm{g}=0$, where $0=\mathbb{1}$, we conclude that the conditions of Theorem~\ref{T-st-sxq-typ-Axy-Bxx-gxh} are fulfilled, and thus the problem under study has regular solutions.

As the next step, we find the minimum value $\theta$ by application of \eqref{E-theta}. The evaluation of $\theta$ involves the matrices
\begin{gather*}
\bm{S}_{0}
=
\bm{I},
\quad
\bm{S}_{1}
=
\bm{A}\oplus\bm{A}\bm{B}\oplus\bm{A}\bm{B}^{2},
\quad
\bm{S}_{2}
=
\bm{A}^{2}\oplus\bm{A}\bm{B}\bm{A}\oplus\bm{A}^{2}\bm{B},
\quad
\bm{S}_{3}
=
\bm{A}^{3},
\\
\bm{T}_{0}
=
\bm{B}^{\ast},
\qquad
\bm{T}_{1}
=
\bm{A}
\oplus
\bm{A}\bm{B}
\oplus
\bm{B}\bm{A},
\qquad
\bm{T}_{2}
=
\bm{A}^{2}.
\end{gather*}

To obtain $\bm{S}_{1}$, $\bm{S}_{2}$ and $\bm{T}_{1}$, we calculate the matrices
$$
\bm{A}\bm{B}
=
\left(
\begin{array}{ccc}
0 & 3 & 5
\\
3 & 1 & 5
\\
2 & 0 & 3
\end{array}
\right),
\qquad
\bm{B}\bm{A}
=
\left(
\begin{array}{crr}
2 & 2 & 4
\\
4 & 3 & 5
\\
3 & -1 & -\infty
\end{array}
\right),
$$
and then the matrices
$$
\bm{A}\bm{B}^{2}
=
\left(
\begin{array}{crc}
4 & -1 & 5
\\
4 & 2 & 4
\\
2 & 1 & 3
\end{array}
\right),
\quad
\bm{A}\bm{B}\bm{A}
=
\left(
\begin{array}{ccc}
6 & 6 & 8
\\
7 & 6 & 8
\\
6 & 4 & 6
\end{array}
\right),
\quad
\bm{A}^{2}\bm{B}
=
\left(
\begin{array}{ccc}
4 & 7 & 9
\\
6 & 5 & 8
\\
5 & 4 & 6
\end{array}
\right).
$$

After substitution of these matrices, we have
$$
\bm{S}_{1}
=
\left(
\begin{array}{ccc}
4 & 3 & 5
\\
4 & 3 & 5
\\
2 & 1 & 3
\end{array}
\right),
\qquad
\bm{S}_{2}
=
\left(
\begin{array}{ccc}
8 & 7 & 9
\\
7 & 6 & 8
\\
6 & 4 & 6
\end{array}
\right),
\qquad
\bm{T}_{1}
=
\left(
\begin{array}{ccc}
4 & 3 & 5
\\
4 & 3 & 5
\\
3 & 1 & 3
\end{array}
\right).
$$

Based on the results obtained, we calculate the sum
$$
\bigoplus_{k=1}^{3}\mathop\mathrm{tr}\nolimits^{1/k}(\bm{S}_{k})
=
4.
$$

To evaluate the remaining sums, we first find the vectors
$$
\bm{h}^{-}\bm{T}_{0}
=
\left(
\begin{array}{ccc}
-2 & -3 & -1
\end{array}
\right),
\qquad
\bm{h}^{-}\bm{T}_{1}
=
\left(
\begin{array}{ccc}
2 & 1 & 3
\end{array}
\right),
\qquad
\bm{h}^{-}\bm{T}_{2}
=
\left(
\begin{array}{ccc}
6 & 3 & 3
\end{array}
\right),
$$
and then obtain
$$
\bm{h}^{-}\bm{T}_{1}\bm{g}
=
4,
\quad
\bm{h}^{-}\bm{T}_{2}\bm{g}
=
6,
\quad
\bm{h}^{-}\bm{T}_{0}\bm{p}
=
4,
\quad
\bm{h}^{-}\bm{T}_{1}\bm{p}
=
8,
\quad
\bm{h}^{-}\bm{T}_{2}\bm{p}
=
10.
$$

With these results, we get another two sums 
$$
\bigoplus_{k=1}^{2}(\bm{h}^{-}\bm{T}_{k}\bm{g})^{1/k}
=
\bigoplus_{k=0}^{2}
(\bm{h}^{-}\bm{T}_{k}\bm{p})^{1/(k+1)}
=
4.
$$

Furthermore, we obtain the vectors
\begin{align*}
\bm{q}^{-}\bm{S}_{0}
&=
\left(
\begin{array}{rrr}
-3 & -2 & -1
\end{array}
\right),
&
\bm{q}^{-}\bm{S}_{1}
&=
\left(
\begin{array}{ccc}
2 & 1 & 3
\end{array}
\right),
\\
\bm{q}^{-}\bm{S}_{2}
&=
\left(
\begin{array}{ccc}
5 & 4 & 6
\end{array}
\right),
&
\bm{q}^{-}\bm{S}_{3}
&=
\left(
\begin{array}{ccc}
9 & 7 & 8
\end{array}
\right),
\end{align*}
and then calculate
\begin{gather*}
\bm{q}^{-}\bm{S}_{1}\bm{g}
=
4,
\qquad
\bm{q}^{-}\bm{S}_{2}\bm{g}
=
7,
\qquad
\bm{q}^{-}\bm{S}_{3}\bm{g}
=
9,
\\
\bm{q}^{-}\bm{S}_{0}\bm{p}
=
4,
\qquad
\bm{q}^{-}\bm{S}_{1}\bm{p}
=
8,
\qquad
\bm{q}^{-}\bm{S}_{2}\bm{p}
=
11,
\qquad
\bm{q}^{-}\bm{S}_{3}\bm{p}
=
13.
\end{gather*}

Finally, we use the above results to find the last two sums
$$
\bigoplus_{k=1}^{3}(\bm{q}^{-}\bm{S}_{k}\bm{g})^{1/k}
=
\bigoplus_{k=0}^{3}(\bm{q}^{-}\bm{S}_{k}\bm{p})^{1/(k+1)}
=
4.
$$

By combining all sums according to \eqref{E-theta}, we have
$$
\theta
=
4.
$$

To describe the solution set defined by \eqref{E-x-y} and \eqref{I-thetapg-u-thetaqAhthetaAB}, we first obtain
$$
\theta^{-1}\bm{q}^{-}\bm{A}
=
\left(
\begin{array}{ccc}
-3 & -3 & -2
\end{array}
\right),
\qquad
\theta^{-1}\bm{q}^{-}\bm{A}\oplus\bm{h}^{-}
=
\left(
\begin{array}{ccc}
-2 & -3 & -2
\end{array}
\right).
$$

We calculate the matrices
$$
\theta^{-1}\bm{A}
\oplus
\bm{B}
=
\left(
\begin{array}{rrr}
0 & -1 & 1
\\
0 & -1 & 2
\\
-1 & -3 & -1
\end{array}
\right),
\qquad
(\theta^{-1}\bm{A}
\oplus
\bm{B})^{2}
=
\left(
\begin{array}{rrc}
0 & -1 & 1
\\
1 & -1 & 1
\\
-1 & -2 & 0
\end{array}
\right),
$$
and then find
$$
(\theta^{-1}\bm{A}
\oplus
\bm{B})^{\ast}
=
\left(
\begin{array}{rrc}
0 & -1 & 1
\\
1 & 0 & 2
\\
-1 & -2 & 0
\end{array}
\right).
$$

With \eqref{E-x-y}, all solutions $\bm{x}=(x_{1},x_{2},x_{3})^{T}$ to the problem are given by
$$
\bm{x}
=
(\theta^{-1}\bm{A}\oplus\bm{B})^{\ast}\bm{u},
\qquad
\bm{u}_{1}
\leq
\bm{u}
\leq
\bm{u}_{2},
$$
where the bounds for the vector $\bm{u}=(u_{1},u_{2},u_{3})^{T}$ in \eqref{I-thetapg-u-thetaqAhthetaAB} are defined as
$$
\bm{u}_{1}
=
\theta^{-1}\bm{p}\oplus\bm{g}
=
\left(
\begin{array}{c}
0
\\
0
\\
1
\end{array}
\right),
\quad
\bm{u}_{2}
=
((\theta^{-1}\bm{q}^{-}\bm{A}\oplus\bm{h}^{-})(\theta^{-1}\bm{A}\oplus\bm{B})^{\ast})^{-1}
=
\left(
\begin{array}{c}
2
\\
3
\\
1
\end{array}
\right).
$$

Note that the columns in the matrix $(\theta^{-1}\bm{A}\oplus\bm{B})^{\ast}$ are equal up to constant factors, and therefore, this matrix can be represented as 
$$
\left(
\begin{array}{rrc}
0 & -1 & 1
\\
1 & 0 & 2
\\
-1 & -2 & 0
\end{array}
\right)
=
\left(
\begin{array}{c}
1
\\
2
\\
0
\end{array}
\right)
\left(
\begin{array}{rrc}
-1
&
-2
&
0
\end{array}
\right).
$$

We introduce a new scalar variable
$$
v
=
\left(
\begin{array}{rrc}
-1
&
-2
&
0
\end{array}
\right)
\bm{u},
$$
and rewrite the solution in the form
$$
\bm{x}
=
\left(
\begin{array}{c}
1
\\
2
\\
0
\end{array}
\right)
v,
\qquad
v_{1}
\leq
v
\leq
v_{2},
$$
where the lower and upper bounds on $v$ are given by
$$
v_{1}
=
\left(
\begin{array}{rrc}
-1
&
-2
&
0
\end{array}
\right)
\bm{u}_{1}
=
1,
\qquad
v_{2}
=
\left(
\begin{array}{rrr}
-1
&
-2
&
0
\end{array}
\right)
\bm{u}_{2}
=
1.
$$

Since both bounds coincide, we have the single vector of initiation time
$$
\bm{x}
=
\left(
\begin{array}{c}
2
\\
3
\\
1
\end{array}
\right).
$$

Finally, using formulas \eqref{E-x-y} and \eqref{E-s-t} gives the vector of completion time and the vectors of adjusted initiation and completion times
$$
\bm{y}
=
\left(
\begin{array}{c}
6
\\
6
\\
4
\end{array}
\right),
\qquad
\bm{s}
=
\left(
\begin{array}{c}
2
\\
2
\\
1
\end{array}
\right),
\qquad
\bm{t}
=
\left(
\begin{array}{c}
6
\\
6
\\
5
\end{array}
\right).
$$

\section*{Acknowledgements}
This work was supported in part by the Russian Foundation for Humanities (grant No. 16-02-00059). The author thanks two referees for valuable comments and suggestions, which have been incorporated into the final version of the manuscript.

\bibliographystyle{abbrvurl}
\bibliography{Direct_solution_to_constrained_tropical_optimization_problems_with_application_to_project_scheduling}

\end{document}